\documentclass{article}
\usepackage[dvipsnames]{xcolor}
\usepackage{oldlfont}
\usepackage{enumerate}
\usepackage{amsmath}
\usepackage{amssymb}
\usepackage{amsthm}
\usepackage{mathrsfs}
\usepackage{amscd}
\usepackage{exscale}
\usepackage{latexsym}
\usepackage[all]{xy}
\usepackage{hyperref}
\usepackage{fancyhdr}
\usepackage{layout}
\usepackage{color}
\usepackage{graphicx}
\usepackage{mathtools}
\usepackage{amsfonts}
\usepackage{calc}
\usepackage{tikz}
\usepackage{mathtools}
\usepackage[OT2,T1]{fontenc}
\usepackage[utf8]{inputenc}

\usepackage[normalem]{ulem}

\usepackage{a4wide}

\DeclareMathAlphabet{\mathpzc}{OT1}{pzc}{m}{it}
\DeclareSymbolFont{cyrletters}{OT2}{wncyr}{m}{n}
\DeclareMathSymbol{\Sha}{\mathalpha}{cyrletters}{"58}

\begin{document}

\baselineskip=17pt

\pagestyle{headings}

\numberwithin{equation}{section}

\makeatletter                                                           

\def\section{\@startsection {section}{1}{\z@}{-5.5ex plus -.5ex         
minus -.2ex}{1ex plus .2ex}{\large \bf}}                                 


\pagestyle{fancy}
\renewcommand{\sectionmark}[1]{\markboth{ #1}{ #1}}
\renewcommand{\subsectionmark}[1]{\markright{ #1}}
\fancyhf{} 
\fancyhead[LE,RO]{\slshape\thepage}
\fancyhead[LO]{\slshape\rightmark}
\fancyhead[RE]{\slshape\leftmark}

\addtolength{\headheight}{0.5pt} 
\renewcommand{\headrulewidth}{0pt} 

\newtheorem{thm}{Theorem}[section]
\newtheorem{mainthm}[thm]{Main Theorem}
\newtheorem*{T}{Theorem 1'}

\newcommand{\ZZ}{{\mathbb Z}}
\newcommand{\GG}{{\mathbb G}}
\newcommand{\Z}{{\mathbb Z}}
\newcommand{\RR}{{\mathbb R}}
\newcommand{\NN}{{\mathbb N}}
\newcommand{\GF}{{\rm GF}}
\newcommand{\QQ}{{\mathbb Q}}
\newcommand{\CC}{{\mathbb C}}
\newcommand{\FF}{{\mathbb F}}

\newtheorem{lem}[thm]{Lemma}
\newtheorem{cor}[thm]{Corollary}
\newtheorem{pro}[thm]{Proposition}
\newtheorem*{proposi}{Proposition \ref{pro:pro63}}
\newtheorem*{thm_notag}{Theorem}
\newtheorem{problem}{Problem}
\newtheorem*{Cassels}{Cassels' question}

\newtheorem{proprieta}[thm]{Property}
\newcommand{\pf}{\noindent \emph{Proof.} \ }
\newcommand{\eop}{${\Box}$  \relax}
\newtheorem{num}{equation}{}

\theoremstyle{definition}
\newtheorem{rem}[thm]{Remark}
\newtheorem{rems}[thm]{Remarks}
\newtheorem{D}[thm]{Definition}
\newtheorem{Not}{Notation}

\newtheorem{Def}{Definition}

\newcommand{\res}{\operatorname{res}}
\newcommand{\nsplit}{\cdot}
\newcommand{\GGG}{{\mathfrak g}}
\newcommand{\GL}{{\rm GL}}
\newcommand{\SL}{{\rm SL}}
\newcommand{\SP}{{\rm Sp}}
\newcommand{\LL}{{\rm L}}
\newcommand{\Ker}{{\rm Ker}}
\newcommand{\la}{\langle}
\newcommand{\ra}{\rangle}
\newcommand{\PSp}{{\rm PSp}}
\newcommand{\Uni}{{\rm U}}
\newcommand{\SUni}{{\rm SU}}
\newcommand{\GU}{{\rm GU}}
\newcommand{\GO}{{\rm GO}}
\newcommand{\Pro}{{\rm P}}
\newcommand{\Aut}{{\rm Aut}}
\newcommand{\Alt}{{\rm Alt}}
\newcommand{\Sym}{{\rm Sym}}
\newcommand{\End}{{\rm End}}

\newcommand{\isom}{{\cong}}
\newcommand{\z}{{\zeta}}
\newcommand{\Gal}{{\rm Gal}}
\newcommand{\SO}{{\rm SO}}
\newcommand{\PSO}{{\rm PSO}}
\newcommand{\PO}{{\rm PO}}
\newcommand{\SU}{{\rm SU}}
\newcommand{\PGL}{{\rm PGL}}
\newcommand{\PSL}{{\rm PSL}}
\newcommand{\loc}{{\rm loc}}
\newcommand{\Sp}{{\rm Sp}}
\renewcommand{\O}{{\rm O}}
\newcommand{\Suz}{{\rm Suz}}
\newcommand{\PUni}{{\rm PU}}
\newcommand{\Id}{{\rm Id}}
\newcommand{\I}{{\rm I}}
\newcommand{\Ind}{{\rm Ind}}
\newcommand{\s}{{\sigma}}

\newcommand{\F}{{\mathbb F}}
\newcommand{\Q}{{\mathbb Q}}
\newcommand{\R}{{\mathbb R}}
\newcommand{\N}{{\mathbb N}}
\newcommand{\E}{{\mathcal{E}}}
\newcommand{\G}{{\mathcal{G}}}
\newcommand{\A}{{\mathcal{A}}}
\newcommand{\C}{{\mathcal{C}}}
\newcommand{\modn}{{\rm \hspace{0.1cm} (mod \hspace{0.1cm} }}
\newcommand{\bmu}{{\textbf \mu}}
\newcommand{\hloc}{{\rm loc}}
\newcommand{\f}{{\mathfrak f}}
\newcommand{\g}{{\mathfrak g}}
\newcommand{\h}{{\mathfrak h}}
\newcommand{\cost}{{\mathfrak c}}

\newcommand\ddfrac[2]{\frac{\displaystyle #1}{\displaystyle #2}}

\let\oldproofname=\proofname
\renewcommand{\proofname}{\rm\bf{\oldproofname}}

\title{Cohomology of groups acting on \\ vector spaces over finite fields}
\author{Davide Lombardo and Laura Paladino}
\date{  }
\maketitle

\begin{abstract}
Let $\FF_q$ be the finite field with $q=p^m$ elements and $G$ be a subgroup of $\GL_n(\FF_q)$.
A famous theorem of Nori published in 1987 states that there
exists a (non-effective) constant $c(n)$, depending only on $n$, such that if $p>c(n)$ and $G$ acts semisimply on $\FF_p^n$, then  $H^1(G,\FF_p^n)=0$. 
We solve the long-standing problem, 
also considered by Serre, of giving an effective proof of Nori's theorem. 
Our approach yields the optimal constant $c(n)=n+2$.
We also prove a more general version of Nori's theorem, namely, that for all powers $q$ of $p$, if $G$ acts semisimply on $\FF_q^n$ and $p>n+2$, then $H^1(G,\FF_q^n)$ is trivial.
We apply these results to refine a criterion, proved by \c{C}iperiani and Stix, which gives sufficient conditions for an affirmative answer to a classical
question posed by Cassels in the case of abelian varieties over number fields.
\end{abstract}

\section{Introduction} \label{sec0}
In the study of group actions, the first cohomology group is an object of fundamental importance, and much work has been devoted to understanding $H^1(G, M)$ for general groups $G$ and $G$-modules $M$. Of particular interest are criteria guaranteeing the vanishing of $H^1(G, M)$, of which there are several:  
for example, $H^1(G, M)$ is known to be trivial
when $G$ is a group of Lie type and $M$ is a minimal irreducible $G$-module \cite{CPS, CPS2} or 
$M=L(\lambda)$ is an irreducible $G$-module of highest weight $\lambda$ 
\cite{Georgia}. 
Let $\FF_q$ be the finite field with $q=p^m$ elements, where $p$ is a prime number and $m$ is a positive integer, and let $G$ be a subgroup of $\GL_n(\F_q)$, for some positive integer $n$. In this work we are concerned with the vanishing of $H^1(G, V)$ for the natural $G$-module $V = \F_q^n$, under the assumption that the action of $G$ on $V$ is semisimple.
A deep theorem of Nori \cite[Theorem E, page 271]{Nori} gives the following sufficient conditions
for the triviality of $H^1(G,\FF_q^n)$, when $q=p$.

\begin{thm}[Nori, 1987] \label{Nori}
Let $p$ be a prime number and let $n$ be a positive integer.
Let $G$ be a subgroup of $\GL_n(\FF_p)$
acting semisimply on $\FF_p^n$. There exists a constant $c(n)$, depending only on $n$, such that if $p>c(n)$, then $$H^1(G,\FF_p^n)=0.$$

\end{thm}

\noindent  In his letter to Marie-France Vign\'eras in 1986,
Serre gives an alternative proof of Nori's theorem and mentions the problem of finding an explicit $c(n)$ (see \cite{Oev}). He also
states that he has no conjectures about the possible value of $c(n)$, as well as no examples \cite[Remark 2, pag.\ 42]{Oev}.
There are no results in the literature which give an explicit $c(n)$ for all $G$.
Only for groups $G$ containing no nontrivial normal $p$-subgroup, Guralnick \cite{Gur} proved in 1999 that one can take $c(n)=n+2$. 
Here we answer Serre's question by showing that one can take $c(n)=n+2$ for every $n$ and $G$. In addition, we prove more generally that Nori's theorem can be extended (with the same optimal constant $c(n)=n+2$)
to subgroups $G\leq\GL_n(\FF_q)$ acting semisimply on $\FF_q^n$, where $q$ is now allowed to be any power of $p$.

\begin{thm} \label{P1_bis}
Let $p$ be a prime number and let $m, n$ be positive integers.
Let $G$ be a subgroup of $\GL_n(\FF_q)$
acting semisimply on $\FF_q^n$, where $q=p^m$. If $p>n+2$, then $$H^1(G,\FF_q^n)=0.$$
\end{thm}

\noindent 
The constant $c(n)=n+2$ is sharp, since there are counterexamples for $p\leq n+2$.
The simplest counterexamples can be obtained from the alternating groups $A_p$, with $p \geq 5$, in dimension $n=p-2$. Let $U=\{(x_1,\ldots,x_p) \in \F_p^p : \sum_{i=1}^p x_i=0\}$ and $W=\{(x,\ldots,x) \in \F_p^p : x \in \F_p\} \subset U$. Clearly, $W \cong \F_p$ is a trivial 1-dimensional $A_p$-module, while, by \cite[p.~186]{KL}, $U/W$ is a simple $A_p$-module of dimension $p-2$. Using \cite[Lemma 5.3.4]{KL}, one checks immediately that the sequence
\[
\{0\} \to W \to U \to U/W \to \{0\}
\]
does not split, which proves $H^1(A_p, U/W) \neq (0)$. Further counterexamples with $n=p-2$ are given by certain projective special linear groups $\textrm{P}\SL_r(w)$, where $w=r^{\alpha}$ is a power of a prime $r\neq p$ and $p$ equals $\frac{w^r-1}{w-1}$ (see the introduction to \cite{Gur}).

\par  The techniques we use to prove Theorem \ref{P1_bis} are different from those employed by Nori. As a first step, we use the aforementioned result of Guralnick \cite{Gur}  to handle the special case when $V = \FF_q^n$ is an irreducible $G$-module (see Proposition \ref{prop: simple plus trivial} for a slightly more general statement). As $V$ is assumed to be a semisimple $G$-module, we can write $V=\bigoplus_{i=1}^k V_i$, where each $V_i$ is an irreducible $G$-submodule of $V$. Since cohomology commutes with finite direct sums, we have $H^1(G,V)=\bigoplus_{i=1}^k H^1(G, V_i)$, where each $V_i$ is a simple $G$-module. However, the $V_i$ need not be \textit{faithful} $G$-modules, hence we cannot simply apply the result for simple modules to every $H^1(G, V_i)$. We circumvent this difficulty by a careful study of the kernels $N_i$ of the natural projections $\GL(V) \to \GL(\bigoplus_{j \neq i}V_j)$ and of their cohomology. Combined with a double induction (on $|G|$ and on $k$, the number of irreducible components of $V$), this leads to a proof of our main theorem. 

\par  The triviality of cohomology groups also has applications beyond group theory. As a first immediate consequence, in Section \ref{sect: supplements} we use Theorem \ref{P1_bis} to obtain structural information about the semisimple subgroups of $\GL_n(\F_q)$ when $p>n+3$  (see Corollary \ref{cor: no normal subgroups of index a power of p}). In the same section  we also make a remark about the vanishing of $H^1$ for groups that decompose nontrivially as a tensor product (Proposition \ref{tensor}). This gives a further criterion for the vanishing of $H^1$ which can be useful in itself, and shows that the constant in Theorem \ref{P1_bis}, though sharp in general, can be greatly improved if one restricts to special classes of subgroups of $\GL_n(\F_q)$.
\par The vanishing of cohomology groups also finds applications in number theory, for instance in local-global questions (which often have an equivalent formulation in terms of principal homogeneous spaces under certain group schemes).  We
give an application of Theorem
\ref{P1_bis} to the following question,
posed by Cassels in 1962, in the third paper of his remarkable series on the arithmetic of curves of genus one  \cite{Cas}.

\begin{Cassels} 
Let $k$ be a number field, with algebraic closure $\overline{k}$, and $\E$ be an elliptic curve defined over $k$. 
Let $G_k$ be the absolute Galois group $\Gal(\overline{k}/k)$ and $\Sha(k,\E)$ be
the Tate-Shafarevich group of $\E$ over $k$. Fix a prime $p$. Are the elements of $\Sha(k,\E)$ infinitely divisible by $p$, when considered as elements of the Weil-Ch\^{a}telet group $H^1(G_k,\E)$ of all classes of principal homogeneous spaces for $\E$ defined over $k$?
\end{Cassels}

\noindent An element $\sigma\in \Sha(k,\E)$ is said to be \textit{infinitely
divisible} by a prime $p$ when it is divisible by $p^m$ for all positive integers $m$, i.e.,
when for all $m \in \mathbb{N}$ there exists $\tau_m\in H^1(G_k,\E)$ such that $\sigma =p^m\tau_m$.

For $m=1$, in 1962 Cassels \cite{Cas4} showed (using a lemma of Tate) that every $\sigma\in \Sha(k,\E)$ is divisible by $p$ in $H^1(G_k,\E)$, but the question remained open for 50 years for powers $p^m$ with $m\geq 2$. Only in 2012 did we get a positive answer for all $p\geq  ( 3^{[k:\Q] /2} + 1 )^2$ and all $m$, implied by the results in \cite{PRV} combined with
\cite[Theorem 3]{Cre2}
(see \cite{DP} for further details and see also \cite[Theorem B]{CS}). In addition,
results from 2014 \cite{PRV2} imply a positive answer over
$\QQ$ for all $p\geq 5$ (see \cite{DP} for more details and \cite{LW} for a second proof of the same result). This bound is best possible for $k=\QQ$: in 2016, Creutz gave counterexamples to the divisibility by $p^m$ for $p=2,3$ and $m\geq 2$ \cite{Cre}.

Since 1972, Cassels' question has also been considered in the more general setting of an abelian variety $\A$ defined over $k$ instead of an elliptic curve \cite{Bas, CS, Cre2, GR}. 
In \cite{Cre2}, Creutz constructed, for every prime $p$, infinitely many non-isomorphic abelian varieties over $\QQ$ for which the answer to (the analogue of) Cassels' question is negative.

Let $\A[p^m]$ be the $p^m$-torsion subgroup of $\A$ and
$\A[p^m]^{\vee}$ be its Cartier dual. The aforementioned result of Creutz \cite[Theorem 3]{Cre2}
implies that the triviality of $\Sha(k,\A[p^m]^{\vee})$,
for every $m$, is a sufficient condition for Cassels' question to have an affirmative answer for $p$ and $\A$ (see also \cite[Proposition 4.3]{CS}).
\c{C}iperiani and Stix \cite[Theorem D]{CS} also found sufficient conditions ensuring the triviality of $\Sha(k,\A[p^m])$ for all $m \geq 1$. In particular, the vanishing of these groups implies an affirmative answer to Cassels' question in the case 
of principally polarized abelian varieties.

As a consequence of Theorem \ref{P1_bis}, we 
refine the criterion of \c{C}iperiani and Stix
and give sufficient conditions on $\A[p]$ which ensure the simultaneous triviality of
 $\Sha(k,\A[p^m])$ and $\Sha(k,\A[p^m]^{\vee})$, and in particular, imply a positive answer to Cassels' question. Notice that this result also applies to non-principally polarised abelian varieties.
\begin{thm} \label{cor_cass}
Let $\A$ be an abelian variety of dimension $g$ defined over a number field $k$.
For any prime $p>2g+2$, if $\A[p]$ is an irreducible $G_k$-module which is not isomorphic to
a subquotient of $\End(\A[p])$, then Cassels' question has an affirmative answer
for $p$ and $\A/k$. 
\end{thm}

\bigskip\noindent On a related topic, in \cite{CS} and \cite{Cre} the authors ask whether
a local-global principle for divisibility by $p^m$ holds in $H^1(G_k,\A)$.
Let $v$ be a place of $k$ and let $k_v$ be the completion of $k$ at $v$.
We denote by $\overline{k_v}$ the algebraic closure of $k_v$ and by $G_{k_v}$
the Galois group $\Gal(\overline{k_v}/k_v)$.
An element $\sigma\in H^1(G_k,\A)$ is locally divisible by $p^m$ over $k_v$ if
there exists $\tau_v\in H^1(G_{k_v},\A)$ such that $\res_v(\sigma)=p^m \tau_v$,
where  $\res_v(\sigma)$ is the image of $\sigma$ under the restriction map $\res_v: H^1(G_k,\A)\rightarrow H^1(G_{k_v},\A)$.
It is natural to ask whether an element locally divisible by $p^m$ for all $k_v$ is 
divisible by $p^m$ in  $H^1(G_k,\A)$, i.e., if 
the  local-global principle for divisibility by $p^m$ holds in  $H^1(G_k,\A)$. 
In Section \ref{last} we show that the hypotheses of Theorem \ref{cor_cass} are also sufficient to ensure the validity of this local-global principle.

\begin{cor} \label{last_cor}
Let $\A$ be an abelian variety of dimension $g$ defined over a number field $k$.
For every prime $p>2g+2$, if $\A[p]$ is an irreducible $G_k$-module which is not isomorphic to a subquotient
of $\End(\A[p])$, then 
the local-global principle for divisibility by $p^m$ holds in $H^1(G_k,\A)$, 
for every $m\geq 1$.
\end{cor}

\section{Notation and preliminaries} \label{sec1}

Let $n, m$ be positive integers, let $p$ be a prime number, and let $V=\FF_q^n$, with $q=p^m$. We will write $\GL_n(q)$ for the group $\GL_n(\F_q) \cong \GL(V)$.
Throughout the paper, when we say that a group $G\leq \GL_n(q)$ acts on $V$ we mean the natural action given by matrix multiplication. 



\subsection{Preliminary lemmas}

We collect here two facts about the invariants of a group acting on a vector space and a criterion for the vanishing of $H^1(G, V)$ in terms of the vanishing of $H^1(H, V)$ for a normal subgroup $H$ of $G$.

\begin{lem}\label{lemma: p-group has nontrivial invariants}
Let $p$ be a prime number, $k$ be a finite field of characteristic $p$, and $G$ be a $p$-group. For every finite-dimensional $k$-vector space $V$ and every linear action of $G$ on $V$ (that is, an action of $G$ on $V$ via $\operatorname{GL}(V)$), the subspace $V^G = \{v \in V : g \cdot v=v \; \forall g \in G \}$ has positive dimension.
\end{lem}
\begin{proof}
Consider the orbit equation
\[
|V| = |V^G| + \sum_{v \in R} |\operatorname{Orbit}(v)|,
\]
where $R$ is a set of representatives for the non-trivial orbits of the action of $G$ on $V$. For $v \in R$ the cardinality
$
|\operatorname{Orbit}(v)| = \frac{|G|}{|\operatorname{Stab}(v)|}
$
is a power of $p$ which is not 1 (since $\operatorname{Stab}(v)$ is a proper subgroup of $G$, and the order of $G$ is a power of $p$). The equation above then gives $|V| \equiv |V^G| \pmod{p}$. We obtain $|V^G| \equiv |V| \equiv |k|^{\dim V} \equiv 0 \pmod p$, so $|V^G|$ cannot consist only of the zero vector.
\end{proof}

\begin{lem}\label{lemma: invariants for normal subgroup}
    Let $G$ be a group acting linearly on the vector space $V$. Let $N$ be a normal subgroup of $G$. The subspace
    \[
    V^N := \{v \in V : n \cdot v = v \; \forall n \in N \}
    \]
    is invariant under the action of $G$. In particular, if $V$ is irreducible and $V^N \neq \{0\}$, then $V^N=V$.
\end{lem}
\begin{proof}
    We check invariance: if $v$ is in $V^N$ and $g$ is any element of $G$, proving that $g \cdot v$ is in $V^N$ amounts to showing that $n \cdot g \cdot v = g \cdot v$ for all $n \in N$. Equivalently, we need to show $(g^{-1} n g) \cdot v = v$ for all $n \in N$. As $N$ is a normal subgroup, $g^{-1} n g$ is an element of $N$ for every $g \in G$, and by construction, $v$ is fixed by any element of $N$, so $(g^{-1} n g) \cdot v = v$ as desired. The rest of the statement follows.
\end{proof}





The next corollary  deals with the case when $G$ has a nontrivial normal subgroup $H$ with trivial cohomology. 

\begin{cor} \label{extension}  
Let $V=\FF_q^n$ and let $G\leq \GL_{n}(q)$ be a subgroup acting irreducibly on $V$. Let
$H$ be a nontrivial normal subgroup of $G$. 
If  $H^1(H,V)=0$, then $H^1(G,V)=0$.
\end{cor}

\begin{proof} 
 Consider the inflation-restriction
exact sequence

\begin{equation} \label{infres} 0\rightarrow H^1(K,V^H)\rightarrow H^1(G,V)\rightarrow H^1(H,V)^{K}, \end{equation}

\noindent where $K=G/H$. Since $H^1(H,V)=0$ by assumption, we have $H^1(G,V)\simeq H^1(K,V^H)$. By Lemma \ref{lemma: invariants for normal subgroup}, the subspace $V^H$ is a $G$-submodule of $V$.
As $V$ is an irreducible $G$-module, $V^H$ is either trivial or equal to $V$, and since $H$ is nontrivial the second case cannot arise. 
Thus, $V^H$ is trivial,
so we obtain $H^1(K,V^H)=0$ and $H^1(G,V)=0$.
\end{proof}

\section{Proof of the main theorem}

We are now ready to prove our main theorem. The proof is by double induction, both on the order of $G$ and on the number of irreducible components in the representation $V$ of $G$. We start with the following `base case' when $V$ has only one non-trivial irreducible component.


\begin{pro}\label{prop: simple plus trivial}
Let $p$ be a prime and $m, n$ be positive integers. Write $q=p^m$ and $V=\F_q^n$. Let $G$ be a subgroup of $\GL_n(q)$ and consider $V$ as a $G$-module through the natural action. Suppose that, as a representation of $G$, $V$ decomposes as $V_1 \oplus V_{\operatorname{triv}}$, where $V_1$ is a non-trivial irreducible representation and $V_{\operatorname{triv}}$ is a subspace (of any dimension, including $0$) on which $G$ acts trivially. If $p > n+2$, then $H^1(G, V)=0$.
\end{pro}

\begin{proof}
    If $G$ has no non-trivial normal $p$-subgroups, the result follows from \cite[Theorem A]{Gur}, so it suffices to show that this hypothesis holds. Suppose by contradiction that $G$ admits a non-trivial normal $p$-subgroup $N$. We can consider $N$ as acting on $V_1$, and by Lemma \ref{lemma: p-group has nontrivial invariants} we have $V_1^N \neq \{0\}$. By Lemma \ref{lemma: invariants for normal subgroup}, this implies $V_1^N=V_1$. This equality implies that $N$ acts trivially on all of $V$, hence that $N=\{\operatorname{id}\}$, contradicting the fact that $N$ is nontrivial.
\end{proof}
\noindent \textbf{Proof of Theorem \ref{P1_bis}.}
    We consider $V=\F_q^n$ as a $G$-module.
    Since $V$ is semisimple, we can write it as $V = V_1 \oplus \cdots \oplus V_k$, where each subspace $V_i$ is a simple $G$-submodule of $V$. Several $V_i$ can be isomorphic to each other and the decomposition $V \cong V_1 \oplus \cdots \oplus V_k$ is not canonical, but we simply fix one such decomposition. Note that the number $k$ of simple summands is a well-defined invariant of $V$, independent of our choice of the simple submodules $V_i$ (by the Jordan-H\"older theorem).
    
    The proof is by double induction, on the order $d$ of $G$ and the number $k$ of irreducible components of $V$. We well-order pairs $(d, k)$ by setting $(d_1, k_1) < (d_2, k_2)$ if $d_1 < d_2$ or $d_1=d_2$ and $k_1 < k_2$. In other words, in the inductive step, we want to reduce either to a strictly smaller group, or to a group of the same order which acts via a representation with fewer irreducible components.

The case $d=1$ is obviously trivial for any $k$ since the trivial group has trivial cohomology. The case of arbitrary $d$ and $k=1$ follows from Proposition \ref{prop: simple plus trivial} (taking $V_{\operatorname{triv}}=(0)$). We can now begin the inductive argument. Notice that, since we have already dealt with the case $k=1$, we will be able to assume $k \geq 2$.

For each $i=1, \ldots, k$ there is a natural projection
\[
\pi'_i : G \to \prod_{j \neq i} \GL(V_j),
\]
given by the restriction of the action to the subspace $W_i := \bigoplus_{j \neq i} V_j$ where we omit the $i$-th irreducible component. Notice that the image $H_i$ of $\pi'_i$ can be identified with a subgroup of $\prod_{j \neq i} \GL(V_j) \subseteq \GL(\bigoplus_{j \neq i} V_j) = \GL(W_i)$, and that $H_i$ is the group through which $G$ acts on $W_i$. In particular, $W_i$ is a semisimple representation of $G$, and equivalently of $H_i$. We denote by $N_i$ the kernel of $\pi'_i$.

By Clifford's theorem \cite[Theorem 1]{Cli}, since $N_i$ is a normal subgroup of $G$ and $V$ is semisimple as a representation of $G$, the restriction of the representation $V$ to $N_i$ is still semisimple. Moreover, by Lemma \ref{lemma: invariants for normal subgroup}, the subspace $V^{N_i}$ is a $G$-submodule of $V$. Notice that by construction we have $W_i \subseteq V^{N_i} \subseteq V$. Since $V/W_i \cong V_i$ is an irreducible representation of $G$, there are only two $G$-submodules of $V$ containing $W_i$, namely $W_i$ itself and $V$.
Therefore, for each $i=1, \ldots, k$ there are two cases, which we label $(a)_i$ and $(b)_i$:
\begin{enumerate}
    \item[$(a)_i$] $V^{N_i}=W_i$; in particular, $N_i$ is not the trivial group.
    \item[$(b)_i$] $V^{N_i}=V$, that is, $N_i$ acts trivially on all of $V$, and therefore $N_i=\{\operatorname{id}\}$.
\end{enumerate}
Consider now, for each $i=1,\ldots,k$, the inflation-restriction sequence corresponding to the normal subgroup $N_i$ of $G$:
\begin{equation}\label{eq: inf-res}
1 \to H^1(G/N_i, V^{N_i}) \xrightarrow{\operatorname{inf}} H^1(G,V) \xrightarrow{\operatorname{res}} H^1(N_i, V)^{G/N_i}.
\end{equation}
We first show that, for every $i$, the cohomology group $H^1(N_i, V)^{G/N_i}$ vanishes. Indeed, in case $(b)_i$ the group $N_i$ is trivial, hence its cohomology vanishes. On the other hand, in case $(a)_i$ we distinguish two subcases:
\begin{enumerate}
    \item $N_i$ is a proper subgroup of $G$. We have already observed that the restriction of $V$ to $N_i$ is a semisimple representation, so we can apply the inductive hypothesis since $|N_i| < |G|$ and obtain as desired that $H^1(N_i, V)=(0)$.
    \item $N_i=G$. By definition of $N_i$, this means that $G$ acts trivially on all $V_j$ with $j \neq i$. Hence, $V$ is the direct sum of the simple representation $V_i$ and of the trivial representation $W_i$. If $V_i$ is also the trivial representation, then $G$ acts trivially on all of $V$ and we are done (since then $N_i=G$ is the trivial group, whose cohomology vanishes). Otherwise, if $V_i$ is a non-trivial simple representation, the cohomology group $H^1(N_i, V)=H^1(G, V)=H^1(G, V_i \oplus W_i)$ vanishes by Proposition \ref{prop: simple plus trivial} since $W_i$ is a trivial representation of $G$.
\end{enumerate}
Thus, the inflation-restriction sequence \eqref{eq: inf-res} gives an isomorphism $H^1(G/N_i, V^{N_i}) \cong H^1(G, V)$ for all $i$, and to conclude the proof it suffices to show that $H^1(G/N_i, V^{N_i})$ vanishes for at least one $i$.
Again we distinguish cases:
\begin{enumerate}
    \item if, for some index $i$, we are in case $(a)_i$, then by definition we have $V^{N_i} = W_i$ and $N_i$ is not the trivial group. We have already observed that the quotient $G/N_i$ can be canonically identified with the image $H_i$ of the natural map $\pi'_i : G \to \prod_{j \neq i} \GL(V_j) \subseteq \GL(W_i)$, so $V^{N_i}$ is the natural module for the subgroup $H_i$ of $\GL(W_i)$. Recall that $V^{N_i} = W_i = \bigoplus_{j \neq i} V_j$ is semisimple for the action of $H_i$. Finally, $p>\dim V + 2 > \dim W_i + 2$.
    Thus, since $|H_i|=|G/N_i| < |G|$, we can apply the inductive hypothesis to deduce $H^1(G/N_i, V^{N_i})=H^1(H_i, V^{N_i})=(0)$, as desired.
    \item suppose instead that for \textit{all} indices $i$ we are in case $(b)_i$, so that $N_i=\{\operatorname{id}\}$ for each $i$. By definition, this means that for each $i$ the projection $G \to \GL\left(\bigoplus_{j \neq i} V_j\right)$ is injective, and therefore that $W_i = \bigoplus_{j \neq i} V_j$ is a faithful, semisimple $G$-module with $k-1$ irreducible components. Moreover, one has $p > \dim V+2 > \dim W_i + 2$.
    For each $i$ we can therefore apply the inductive hypothesis to deduce that $H^1(G, W_i)=(0)$. By additivity of the cohomology functor, this gives 
    \[
    (0) = H^1(G, \bigoplus_{j \neq i} V_j) = \bigoplus_{j \neq i} H^1(G, V_j),
    \]
    hence $H^1(G, V_j)=(0)$ for all $j \neq i$. But since this holds for each $i$, we deduce $H^1(G, V_j)=(0)$ for all $j$ (here we use $k \geq 2$), hence 
    \[
    H^1(G, V)= H^1(G, \bigoplus_{i=1}^k V_i) = \bigoplus_{i=1}^k H^1(G, V_i) = (0),
    \]
    as desired. \eop
\end{enumerate}

\section{Two supplements}\label{sect: supplements}
\noindent We begin by proving a direct consequence of Theorem \ref{P1_bis} concerning semisimple subgroups of $\GL_n(q)$.

\begin{cor}\label{cor: no normal subgroups of index a power of p}
Let $n, m$ be positive integers and let $p$ be a prime greater than $n+3$. Let $G$ be a subgroup of $\GL_n(q)$, where $q=p^m$. If $G$ acts semisimply on $\F_q^n$, then $G$ has no normal subgroup of index $p^k$, for any $k \geq 1$. In particular, $G$ admits no non-trivial homomorphism to a $p$-group.
\end{cor}


\begin{proof}
    Let $V=\F_q^n$ be the natural module for the action of $G$ and consider the semisimple representation of $G$ given by $V' = V \oplus \F_q$, where $G$ acts trivially on $\F_q$. Since $\dim V'=n+1$, Theorem \ref{P1_bis} shows $H^1(G,V) \oplus H^1(G, \F_q) = H^1(G, V')=0$, which implies in particular $H^1(G, \F_q) = \operatorname{Hom}(G, \F_q)=0$. It follows that $G$ admits no non-trivial homomorphisms towards $C_p$, the cyclic group of order $p$. If $N$ were a normal subgroup of $G$ of index $p^k$ for some $k \geq 1$, then the $p$-group $G/N$ would admit $C_p$ as a quotient, so $G$ would also admit $C_p$ as a quotient, contradiction.
\end{proof}

\noindent We also consider tensor product groups $G=G_1 \otimes \cdots \otimes G_t$, where each $G_i$ is a subgroup of $\GL(V_i) \cong \GL_{d_i}(q)$ for some $d_i$. The group $G$ acts on $V = V_1 \otimes \cdots \otimes  V_t$, and it is a natural expectation that the triviality of all the groups $H^1(G_i,V_i)$ implies $H^1(G,V)=0$.
Assuming that $V$ is irreducible, we show the stronger statement that it suffices that one of the groups $H^1(G_i,V_i)$ (with $G_i$ nontrivial) 
vanishes to get $H^1(G,V)=0$. In particular, this shows that for this special class of groups the bound of Theorem \ref{P1_bis} can be greatly improved.


\begin{pro} \label{tensor}
Let $V=\bigotimes_{i=1}^t V_i$, where each $V_i$ is a $\F_q$-vector space of dimension $d_i$, for every $1 \leq i \leq t$. Assume that $G=G_1 \otimes \cdots \otimes G_t$ acts on $V$, where $G_i$ is a nontrivial subgroup of $\GL(V_i)$. Assume further that $V$ is an irreducible $G$-module.  If, for some $1 \leq i \leq t$, the group $H^1(G_i, V_i)$ vanishes, then $H^1(G,V)=0$.
 In particular, if $p> (\dim V)^{1/t} + 2$, then $H^1(G, V) = 0$.
\end{pro}

\begin{proof} 
Let $n := \prod_{i=1}^t d_i$. 
By \cite[\S 4.4, pag.\ 129]{KL},
$V$ is a homogeneous $G_i$-module, i.e., $V=\bigoplus _{j=1}^{\frac{n}{d_i}} W_j$, where
$W_j$ is an irreducible $G_i$-module isomorphic to $V_i$, for every $1\leq j\leq \frac{n}{d_i}$.
Thus, if $G_i$ is nontrivial and $H^1(G_i, V_i)=0$, we have
$H^1(G_i, V) = H^1(G_i, \bigoplus _{j=1}^{\frac{n}{d_i}} W_j) \cong \bigoplus _{j=1}^{\frac{n}{d_i}} H^1(G_i, V_i)=0$. The triviality of $H^1(G,V)$ then follows from Corollary \ref{extension}, because $G_i$ is normal in $G$. 

Since the action of $G$ on $V$ is irreducible, the action of each $G_i$ on the corresponding space $V_i$ is also irreducible, hence $H^1(G_i, V_i)=0$ for all $p> \dim V_i + 2$ by Theorem \ref{P1_bis}. In particular, letting $i$ be the index for which $\dim V_i$ is minimal, we have $\dim V \geq (\dim V_i)^t$, and the final claim follows.
\end{proof}





\section{An application to a question of Cassels} \label{last}
Let $\A$ be an abelian variety defined over a number field $k$ and let $\A^{\vee}$ be its dual abelian variety. We denote by $\overline{k}$ the algebraic closure of $k$
and by $G_k$ the absolute Galois group $\Gal(\overline{k}/k)$.
Let $\A[p^m]$ be the $p^m$-torsion subgroup of $\A$ and $\A[p^m]^{\vee}$ be its Cartier dual. 
As recalled in the introduction, an element $\sigma\in \Sha(k,\A)$ is said to be \textit{infinitely
divisible} by a prime $p$ in $H^1(G_k,\A)$ when it is divisible by $p^m$ for all positive integers $m$.
In this case, for every $m$, there exists $\tau_m\in H^1(G_k,\A)$ such that $\sigma=p^m \tau_m$.
Creutz \cite[Theorem 3]{Cre2} showed that $\Sha(k,\A)\subseteq p^mH^1(G_k,\A)$ if and only
if the image of the natural map  $\Sha(k,\A[p^m]^{\vee})\rightarrow \Sha(k,\A^{\vee})$
is contained in the maximal divisible subgroup of $\Sha(k,\A^{\vee})$. 
In particular, the vanishing of  $\Sha(k,\A[p^m]^{\vee})$, for all $m\geq 1$, is a sufficient condition
for an affirmative answer to Cassels' question for $p$ and $\A/k$ (see also \cite[Proposition 4.3]{CS}).
If $\A$ has a principal polarization, then $\A[p^m]$ and $\A[p^m]^{\vee}$ are isomorphic. Therefore
the triviality of  $\Sha(k,\A[p^m])$, for all $m\geq 1$, implies an affirmative answer to Cassels'
question for principally polarized abelian varieties. \c{C}iperiani and Stix \cite{CS} gave the following sufficient condition for the vanishing of $\Sha(k,\A[p^m])$ for every $m\geq 1$. 

\begin{thm}[\c{C}iperiani, Stix, 2015] \label{CS} 
Let $G$ be the Galois group of the finite extension $k(\A[p])/k$, where $k(\A[p])$ denotes the $p$-division field of $\A$ over $k$.
Assume that \begin{description}
\item[1)] $H^1(G,\A[p])=0$  and 
\item[2)] the $G_k$-modules $\A[p]$ and $\textrm{End}(\A[p])$ have no common irreducible subquotient.
\end{description}

\noindent Then $$\hspace{0.5cm} \Sha(k,\A[p^m])=0, \textrm{ for every } m\geq 1.$$
\end{thm}
\noindent Combining this result with Theorem \ref{P1_bis} we obtain the following criterion. 
\begin{thm} \label{thm_cass}
Let $\A$ be an abelian variety of dimension $g$ defined over a number field $k$.
For all $p>2g+2$, if $\A[p]$ is an irreducible $G_k$-module and $\End(\A[p])$ has no subquotient
isomorphic to $\A[p]$,  then
 $$\hspace{0.5cm} \Sha(k,\A[p^m])=0, \textrm{ for every } m\geq 1.$$ 
\end{thm}

\begin{proof}
It is well known that if $\A$ has dimension $g$, then $\A[p]\simeq (\ZZ/p\ZZ)^{2g}\simeq \FF_p^{2g}$
(see for example \cite[Chap. II, pag.\ 61]{Mum}). Thus, up to the choice of an $\F_p$-basis of $\A[p]$, the group $G$ has a natural faithful representation in
$\GL_{2g}(p)$. 
By Theorem \ref{P1_bis}, if $\A[p]$ is an irreducible $G_k$-module (hence an irreducible $G$-module)
and $p>2g+2$, then $H^1(G,\A[p])=0$ and \textbf{1)} in Theorem \ref{CS} is satisfied. 
Since we have assumed that $\A[p]$ is irreducible and that $\End(\A[p])$ has no subquotient
isomorphic to $\A[p]$,  then
$\A[p]$ and $\textrm{End}(\A[p])$ do not have a common irreducible
subquotient and \textbf{2)} is also satisfied. Theorem \ref{CS} now implies $\Sha(k,\A[p^m])=0$ for all $m\geq 1$.
\end{proof}

\noindent From Theorem \ref{thm_cass} we now derive Theorem \ref{cor_cass}, which gives sufficient conditions
for an affirmative answer to Cassels' question for $p$ and $\A/k$, even in the case of non-principally polarised abelian varieties $\A$. 

\bigskip
\noindent \textbf{Proof of Theorem \ref{cor_cass}.}
Let $p>2g+2$. Suppose that $\A[p]^{\vee}$ has a $G_k$-submodule $W$ which is stable under the action
of $G_k$. Since $\A[p]$ is an irreducible $G_k$-module, 
the dual $W^{\vee}$ of $W$ is either isomorphic to $\A[p]$
or it is trivial. Consequently, $W$ is trivial or it is the whole $\A[p]^{\vee}$. Hence
$\A[p]^{\vee}$ is an irreducible $G_k$-module.  
Furthermore, it is easy to see that if $S$ is a subquotient of $\End(\A[p])$, then
$S^{\vee}$ is a subquotient of $\End(\A[p])^{\vee}$. 
Since $\End(\A[p])\simeq\End(\A[p]^{\vee})\simeq\End(\A[p])^{\vee}$, we have that $\A[p]$ is isomorphic to a subquotient of
$\End(\A[p])$ if and only if $\A[p]^{\vee}$ is isomorphic to a subquotient of
$\End(\A[p]^{\vee})$. Therefore, $\End(\A[p])$ has
no subquotient
isomorphic to $\A[p]$ if and only if
$\textrm{End}(\A[p]^{\vee})$ has no subquotient
isomorphic to $\A[p]^{\vee}$.  
Thus, we can apply Theorem \ref{thm_cass} to $\A[p]^{\vee}$ and obtain $\Sha(k,\A[p^m]^{\vee})=0$, for every $m\geq 1$ and $p>2g+2$.
By \cite[Theorem 3]{Cre2}, this shows that
Cassels' question has a positive answer for $p$ and $\A$ (and also for $p$ and $\A^{\vee}$, by applying Theorem \ref{thm_cass} and \cite[Theorem 3]{Cre2} to $\A[p]$). \eop

\bigskip
\noindent \textbf{Proof of Corollary \ref{last_cor}.}
As in the proof of Theorem \ref{cor_cass}, we obtain $\Sha(k,\A[p^m]^{\vee})=0$ for all $p>2g+2$ and all $m \geq 1$.
By \cite[Theorem 2.1]{Cre}, for every positive integer $n$, the triviality of $\Sha(k,\A[n]^{\vee})$ is
a sufficient condition for the validity of the local-global principle for  divisibility by $n$
in $H^1(G_k, \A)$ (see also \cite[Theorem C]{CS}). The corollary follows. \eop

\bigskip\noindent \emph{Acknowledgments}. We thank an anonymous referee
for pointing out some mistakes in an earlier version of this paper.
We warmly thank Gabriele Ranieri, Florence Gillibert, John van Bon
and Rocco Chiriv\`i for useful discussions. We are grateful to Alessandra Caraceni for some helpful remarks. Davide Lombardo and Laura Paladino are members of INdAM-GNSAGA. Lombardo acknowledges funding from the University of Pisa via PRA 2022-10 ``Spazi di moduli, rappresentazioni e strutture combinatorie''.

\bigskip

Davide Lombardo\par\smallskip
Dipartimento di Matematica\par
Universit\`a di Pisa \par
Largo Bruno Pontecorvo, 5\par
56127 Pisa (PI)\par
Italy\par 
e-mail address: davide.lombardo@unipi.it

\vskip 0.5cm

Laura Paladino\par\smallskip
Dipartimento di Matematica e Informatica\par
Universit\`a della Calabria \par
Ponte Bucci, Cubo 30B\par
87036 Rende (CS)\par
Italy\par 
e-mail addresses: laura.paladino@unical.it


\end{document}